\documentclass{amsart}
\usepackage{amsfonts,amssymb,amscd,amsmath,enumerate,verbatim,calc}
\usepackage{xy}

\newcommand{\CM}{Cohen-Macaulay}

\newcommand{\wrt}{with respect to}

\newcommand{\n}{\mathfrak{n} }
\newcommand{\m}{\mathfrak{m} }

\newcommand{\rt}{\rightarrow}

\newcommand{\ov}{\overline}

\newcommand{\low}{\ell\ell}

\newcommand{\vt}{\vartheta }

\newcommand{\bx}{\mathbf{x}}

\newcommand{\reg}{\operatorname{reg}}

\newcommand{\depth}{\operatorname{depth}}

\newcommand{\projdim}{\operatorname{projdim}}
\newcommand{\ind}{\operatorname{index}}

\newcommand{\ord}{\operatorname{ord}}

\theoremstyle{plain}

\newtheorem{theorem}{Theorem}[section]

\newtheorem{lemma}[theorem]{Lemma}

\theoremstyle{definition}

\newtheorem{s}[theorem]{}

\newtheorem{example}[theorem]{Example}

\theoremstyle{remark}

\begin{document}

\title[Lowey length]{On the Lowey length of modules of finite projective dimension.}
\author{Tony~J.~Puthenpurakal}
\date{\today}
\address{Department of Mathematics, IIT Bombay, Powai, Mumbai 400 076}

\email{tputhen@math.iitb.ac.in}
 \begin{abstract}
Let $(A,\m)$ be a local Gorenstein local ring and let $M$ be an $A$ module of finite length and finite projective dimension. We prove that the Lowey length of $M$ is greater than or equal to order of $A$. This generalizes a result of Avramov, Buchweitz, Iyengar and Miller \cite[1.1]{ABIM}.
\end{abstract}
 \maketitle
\section{introduction}
Let $(A,\m)$ be a local Gorenstein local ring  of dimension $d \geq 0$ and embedding dimension $c$. If $M$ is an $A$-module then we let $\lambda(M)$ denote its length. If $A$ is singular then the \emph{order} of $A$ is given by the formula
\[
\ord(A) = \min \big{\{}n\in \mathbb{N} \mid \lambda(A/\m^{n+1}) < \binom{n + c}{n} \big{ \}}.
\]
If $A$ is regular we set $\ord A = 1$. 
Note that if $A$ is singular then $\ord(A) \geq 2$.
Recall that \emph{Lowey length} of an $A$-module $M$ is defined to be the number
\[
\low(M) = \min \{ i \geq 0 \mid \m^i M = 0 \}.
\]
When $M$ is finitely generated $\low(M)$ is finite if and only if $\lambda(M)$ is finite.  Often the Lowey length of $M$ carries more structural information than does it length.

Let $G(A) = \bigoplus_{n\geq 0} \m^n/\m^{n+1}$ be the associated graded ring of $A$ and let $G(A)_+$ denotes its irrelevant maximal ideal. Let $H^i(G(A))$ be the $i^{th}$-local cohomology module of $G(A)$ with respect to $G(A)_+$. The \emph{Castelnuovo-Mumford regularity} of $G(A)$ is
\[
\reg G(A) = \max\{ i + j \mid H^i(G(A))_j \neq 0 \}
\] 
In the very nice paper \cite[1.1]{ABIM} the authors proved that if $G(A)$ is \CM \ then for each non-zero  finitely generated $A$-module $M$ of finite projective dimension
\[
\low(M) \geq \reg G(A) + 1 \geq \ord(A). 
\]
We should note that the real content of their result is that $\low(M) \geq \reg G(A)  + 1$. The fact that $\reg(G(A)) + 1 \geq \ord(A)$ is elementary, see \cite[1.6]{ABIM}. The hypotheses $G(A)$ is \CM \ is quite strong, for instance $G(A)$ need not be \CM \ even if $A$ is a complete intersection.  In this short paper we show
\begin{theorem}\label{main}
Let $(A,\m)$ be a Gorenstein local ring and let $M$ be a non-zero finitely generated module of finite projective dimension. Then
\[
\low(M) \geq  \ord(A). 
\]
\end{theorem}

The proof of Theorem \ref{main} uses invariants of Gorenstein local rings defined by 
Auslander and studied by S. Ding. We also introduce a new invariant $\vt(A)$ which is useful in the case $G(A)$ is not \CM. 

We now describe the contents of this paper in brief. In section two we recall the notion of index of a local ring. In section three we introduce our invariant $\vt(A)$.  In section four we prove Theorem \ref{main}.

\textit{Acknowledgment:} I thank Prof L. L. Avramov for many discussions.
\section{The  index of a Gorenstein local ring}
Let $(A,\m)$ be a  Gorenstein local ring and let $M$ be a finitely generated $A$-module. Let $\mu(M)$ denote minimal number of generators of $M$. In this section we recall the definition of the delta invariant of $M$.  Finally we recall the definition of index of $A$. A good reference for this topic is \cite{LW}.

\s A maximal \CM \ approximation of $M$ is a short exact sequence
\[
0 \rt Y_M \rt X_M \xrightarrow{f} N \rt 0,
\]
where $X_M$ is a maximal \CM \ $A$-module and $\projdim Y_M < \infty$. If $f$ can only be factored through itself by way of an automorphism of $X_M$, then the approximation is said to be minimal. Any module has a minimal approximation and  minimal approximations are unique upto non-unique isomorphisms. Suppose now that $f$ is a minimal approximation. Let $X_M  = \overline{X_M}  \oplus F$ where $\ov{X_M}$ has no free summands and $F$ is free. Then $\delta_A(M)$ is defined to be the rank of $F$.

\s \emph{Alternate definitions of the delta invariant}\\
 It can be shown that $\delta_A(M)$ is  the smallest integer $n$ such that there is an epimorphism $X \oplus A^n \rt M$ with $X$ a maximal \CM  \ module with no free summands, see \cite[4.2]{SS}. This definition of delta is used by Ding \cite{D}.
 
 The stable CM-trace of 
$M$ is the submodule $\tau(M)$ of $M$ generated by the homomorphic images in $M$ of
all  MCM modules without a free summand. Then $\delta_A(M) = \mu(M/\tau(M))$, see \cite[4.8]{SS}. This is the definition of delta in \cite{ABIM}.

\s\label{prop} We collect some properties of the delta invariant that we need. Let $M$ and $N$ be finitely generated $A$-modules.

\begin{enumerate}
\item
If $N$ is an epimorphic image of $M$ then $\delta_A(M) \geq \delta_A(N)$.
\item
$\delta_A(M\oplus N) = \delta_A(M) + \delta_A(N)$.
\item
$\delta_A(M) \leq \mu(M)$. 
\item
If $\projdim M < \infty $ then $\delta_A(M) = \mu(M)$.
\item
Let $x \in \m$ be $A\oplus M$ regular. Set $B = A/(x)$. Then  $\delta_A(M) = \delta_B(M/xM)$.
\item
If $A$ is zero-dimensional Gorenstein local ring and $I$ is an ideal in $A$ then $\delta_A(A/I) \neq 0$ if and only if $I = 0$.
\item
If $A$ is not regular then $\delta_A(\m^s) = 0$ for all $s \geq 1$.
\item
$\delta_A(k) = 1$ if and only if $A$ is regular.
\item
$\delta_A(A/\m^n) \geq 1 $ for all $n \gg 0$.
\end{enumerate}
\emph{Proofs and references} 
For (1),(2),(4),(8),(9); see \cite[1.2]{ABIM}.
Notice (3) follows easily from the second definition of delta.  The assertion (5) is proved in \cite[5.1]{AD}. For (6) note that $A/I$ is maximal \CM. The assertion (7) is due to  Auslander. Unfortunately this paper of Auslander is unpublished. However there is an extension of the delta invariant to all Noetherian local rings due to Martsinkovsky \cite{M}. In a later paper  he proves that $\delta_A(\m) = 0$. see \cite[Theorem 6]{M-r}. We prove by induction that $\delta_A(\m^s) = 0$ for all $s \geq 1$. For $s= 1$ this is true. Assume for $s = j$. We prove it for $s = j+1$. Let $\m^{j+1} = < a_1b_1, a_2b_2,\ldots,a_mb_m>$ where $a_i \in \m^j$ and $b_i \in \m$. Let $I_i = a_i\m $ for $i = 1,\ldots,m$. Note $I_i \subseteq \m^{j+1}$ and the natural map
$\phi \colon \bigoplus_{i = 1}^{m}I_i \rt \m^{j+1}$ is surjective. By (1) and (2) it is enough to show that $\delta_A(I_i) = 0$ for all $i$. But this is clear as $I_i$ is a homomorphic image of $\m$.

\s The \emph{index} of $A$ is defined by Auslander to be the number
\[
\ind(A) = \min \{ n \mid \delta_A(A/\m^n) \geq 1 \}.
\]
It is a positive integer by \ref{prop}(9) and equals 1 if and only if $A$ is regular, see \ref{prop}(8).

\s\label{ineq} By \cite[1.3]{ABIM} we have that if $\projdim M$ is finite then 
\[
\low(M) \geq \ind(A).
\]

\section{The invariant $\vt(A)$}
Throughout this section $(A,\m)$ is a \CM \ local ring  of dimension $d$. We assume that $k$, the residue field of $A$, is  infinite. In this section we define an invariant $\vt(A)$. This will be useful when $G(A)$ is not \CM.

\s Let $a \in A$ be non-zero. Then there exists $n \geq 0$ such that $a \in \m^n \setminus \m^{n+1}$. Set $a^* = $ image of $a$ in $\m^n/\m^{n+1}$ and we consider it as a element in $G(A)$. Also set $0^* = 0$.  If $a \in \m$ is such that $a^*$ is $G(A)$-regular then $G(A/(a)) = G(A)/(a^*)$. 
\s Recall $x \in \m$ is said to be $A$-\textit{superficial} if there exists integer $c > 0$ such that
for $n \gg 0$ we have $(\m^{n+1}\colon x)\cap \m^c = \m^{n-1}$. Superficial elements exist if $d >0$ as $k$ is an infinite field. As $A$ is \CM, it is easily shown that
a superficial element is a non-zero divisor of $A$. Furthermore we have
\[
(\m^{n+1} \colon x) = \m^n \quad \text{for all} \ n \gg 0.
\]
This enables to define the following two invariants of $A$ and $x$:
\[
\vt(A,x) = \inf\{ n \mid (\m^{n+1} \colon x) \neq \m^n \},
\]
\[
\rho(A,x) = \sup\{n \mid (\m^{n+1} \colon x) \neq \m^n \}.
\]
\s \label{obs-1}
Notice $(\m^{n+1} \colon x) = \m^n $ for all $n \geq 0$ if and only if $x^*$ is $G(A)$-regular. Thus in this case $\vt(A,x) = + \infty$ and
$\rho(A,x) = -\infty$. 

If $\depth G(A) > 0$ then $x^*$ is $G(A)$-regular, see \cite[2.1]{HM}.  Thus in this case $\vt(A,x) = + \infty$ and
$\rho(A,x) = -\infty$. 

If $\depth G(A) = 0$ then $(\m^{n+1} \colon x) \neq \m^n $ for some $n$.
In this case we have $\vt(A,x), \rho(A,x)$  are finite numbers and clearly  $\vt(A,x) \leq \rho(A,x)$. By \cite[2.7 and 5.1]{P} we have 
\[
\rho(A,x) \leq \reg G(A) -1.
\]

\s A sequence $\bx = x_1,\ldots,x_r$ in $\m$ with $r \leq d$ is said to be an $A$-\textit{superficial sequence} if $x_i$ is $A/(x_1,\ldots,x_{i-1})$-superficial for $i = 1,\ldots,r$. As the residue field of $A$ is infinite, superficial sequences exist for all $r \leq d$. As $A$ is \CM \ it can be easily shown that superficial sequences 
are regular sequences. 

\s Let $\bx = x_1,\ldots,x_d$ be a maximal $A$-superficial sequence.  Set $A_0 = A$ 
and $A_i = A/(x_1,\ldots,x_i)$ for $i = 1,\ldots,d$. 
Define
\[
\vt(A,\bx) = \inf\{ \vt(A_i,x_{i+1}) \mid  0\leq i \leq d-1 \}.
\]
Note that $G(A)$ is \CM \ if and only if $x_1^*,\ldots,x_d^*$ is a $G(A)$-regular sequence, see \cite[2.1]{HM}. It follows from \ref{obs-1} that 
\[
\vt(A,\bx) =  + \infty \quad \text{if and only if} \ G(A) \ \text{is \CM.} 
\]
We have
\begin{lemma}\label{bound}(with hypotheses as above).
If $G(A)$ is not \CM  \ then $\vt(A,\bx)  \leq \reg G(A) -1$.
\end{lemma}
\begin{proof}
Suppose $\depth G(A) = i < d$. Then  $x_1^*,\ldots,x_i^*$ is $G(A)$-regular, see \cite[2.1]{HM}. Furthermore
$G(A_i)= G(A)/(x_1^*,\ldots,x_i^*)$. Thus
$\depth G(A_i) = 0$. (Note the case $i = 0$ is also included). 

By \ref{obs-1} we have that $\vt(A_i,x_{i+1}) \leq \reg(G(A_i)) -1$. It remains to note that as $x_1^*,\ldots,x_i^*$ is a  regular sequence  of elements of degree $1$ in $G(A)$, we have
 $\reg G(A_i)\leq \reg G(A)$. 
\end{proof}

\s We define
\[
\vt(A) = \sup \{ \vt(A,\bx) \mid \bx \ \text{is a maximal superficial sequence in } \ A \}.
\]
Note that if $G(A)$ is not \CM \ then  $\vt(A) \leq \reg G(A) - 1$, see \ref{bound}. If $G(A)$ is \CM \ then $\vt(A) = + \infty$.

\s Let $A$ be a singular ring and Let $x \in \m$ be an $A$-superficial element.  Let $t = \ord(A)$. The following fact is well-known
(for instance see \cite[p.\ 295]{RV}) 
\[
(\m^{i+1} \colon x) =  \m^i \quad \text{for} \ i = 0,\ldots, t-1.
\]
It follows that $\vt(A,x) \geq \ord(A)$ for any superficial element $x$ of $A$. 

Notice $\ord(A /(x)) \geq \ord(A)$ for any  superficial element $x$ of $A$ (for instance see \cite[p.\ 296]{RV}). Thus if $\bx = x_1,\ldots,x_d$ is a maximal $A$-superficial sequence we have that $\vt(A_i,x_{i+1}) \geq \ord(A_i) \geq \ord(A)$ for all $i = 0,\ldots,d-1$. It follows that
\begin{equation}\label{eqn}
\vt(A,\bx) \geq \ord(A).
\end{equation}

 It is possible that  for some rings strict inequality in \ref{eqn} can hold.
 \begin{example}
 Let $(A,\m)$ be an one dimensional stretched Gorenstein local ring, i.e., there exists an $A$-superficial element $x$ such that if $\n$ is the maximal ideal of $B = A/(x)$ then $\n^2$ is principal. For such rings  $\ord(B) = 2$, see \cite[1.2]{Sa}. So $\ord(A) = 2$. However  for stretched  Gorenstein rings $(\m^3 \colon x) = \m^2$; see \cite[2.5]{Sa}. (Note $(\m^{i+1} \colon x) = \m^i$ for $i \leq 1$ for any \CM \ ring $A$). Thus $\vt(A,x) \geq 3$. 
 
See \cite[Example 3]{Sa} for an example of a stretched one dimensional stretched Gorenstein local ring $A$ with $G(A)$ not \CM.
 \end{example} 
The following result is crucial in the proof of our main result. By $e(A)$ we denote the multiplicity of $A$ \wrt \ $\m$.
\begin{lemma}\label{crucial}
Let $(A,\m)$ be a $d$-dimensional \CM \ local ring with infinite residue field. Let $\bx = x_1,\ldots,x_d$ be a maximal superficial sequence. Assume $G(A)$ is not \CM. Let $n \leq \vt(A,\bx)$. Then
\[
\m^n \nsubseteq (\bx).
\]
\end{lemma}
\begin{proof}
Suppose if possible $\m^n \subseteq (\bx)$ for some $n \leq \vt(A,\bx)$. Set $A_d = A/(\bx)$ and
$A_{d-1} = A/(x_1,\ldots,x_{d-1})$.  Let $\n$ be the maximal ideal of $A_{d-1}$. Notice $n \leq \vt(A_{d-1},x_d)$.

 We have an exact sequence
\[
0 \rt \frac{(\n^{n} \colon x)}{\n^{n-1}} \rt \frac{A_{d-1}}{\n^{n-1}} \xrightarrow{\alpha} \frac{A_{d-1}}{\n^n} \rt \frac{A_{d-1}}{(\n^n, x_d)} \rt 0.
\]
Here $\alpha( a + \n^{n-1}) = ax_d + \n^{n}$. Note that as $\m^n \subseteq (\bx)$ we have $A_{d-1}/(\n^n, x_d) = A_d$. Recall $(\n^{i+1} \colon x_d) = \n^i$ for all $i < \vt(A_{d-1},x_d)$. In particular we have $(\n^{n} \colon x) = \n^{n-1}$.
Thus we have 
\[
\lambda(\n^{n-1}/\n^n) = \lambda(A_d).
\]
Notice $e(A) = e(A_{d-1}) = e(A_d) = \lambda(A_d)$, cf., \cite[Corollary 11]{Pf}.
Furthermore for all $i \geq 0$ we have 
\[
\lambda(\n^i/\n^{i+1}) = e(A_{d-1}) - \lambda(\n^{i+1}/x\n^i), \quad \text{cf., 
\cite[Proposition 13 ]{Pf}}.
\]
For $i = n-1$ our result implies $\n^n = x_d\n^{n-1}$. It follows that $\n^j = x_d\n^{j-1}$ for all $j \geq n$. In particular we have $(\n^j \colon x_d) = \n^{j-1}$ for all $j \geq n$. As $n \leq \vt(A_{d-1},x_d)$ we have that $(\n^j \colon x_d) = \n^{j-1}$ for all $j \leq n$. It follows that $x_{d}^*$ is $G(A_{d-1})$-regular. So $\depth G(A_{d-1}) = 1$. By Sally descent, see \cite[Theorem 8]{Pf} we get $G(A)$ is \CM. This is a contradiction.
\end{proof}

\section{Proof of Theorem \ref{main}}
In this section we prove our main Theorem. We will use the invariant $\vt(A)$ which is defined only when the residue field of $A$ is infinite. We first show that to prove our result we can assume that the residue field of $A$ is infinite.

\s\label{infinite}
Suppose the residue field of $A$ is finite. Consider the flat extension $B = A[X]_{\m A[X]}$. Note that $\n = \m B$ is the maximal ideal of $B$ and $B/\n = k(X)$ is an infinite field. Let $M$ be a finitely generated $A$-module. The following facts can be easily proved:
\begin{enumerate}
\item
$\lambda_B(M \otimes_A B) = \lambda_A(B)$.
\item
$\m^i\otimes B = \n^i$ for all $i \geq 1$.
\item
$\lambda_B(B/\n^{i+1}) = \lambda_A(A/\m^{i+1})$ for all $i \geq 0$.
\item
$\ord(B) = \ord(A)$.
\item
$\projdim_A M = \projdim_B M\otimes_A B$.
\item
$\m^i M = 0$  if and only if $\n^i(M \otimes_A B) = 0$.
\item
$\low_A(M) = \low_B(M\otimes_A B)$.
\end{enumerate}

We need the following result due to Ding, see \cite[2.2,2.3,1.5]{D}.
\begin{lemma}\label{ding}
Let $(A,\m)$ be a Noetherian local ring and $s$ an integer. Suppose that $x \in \m \setminus \m^2$ is $A$-regular and the induced map $\ov{x} \colon \m^{i-1}/\m^{i} \rt \m^{i}/\m^{i+1}$ is injective for $1\leq i \leq s$. Then
\begin{enumerate}[\rm(1)]
\item
$A/\m^s$ is an epimorphic image of $(\m^s,x)$.
\item 
There is an $A$-module decomposition
\[
\frac{(\m^s,x)}{x(\m^s,x)} \cong \frac{A}{(\m^s,x)} \oplus \frac{(\m^s, x)}{(x)}.
\]
\end{enumerate}
\end{lemma}
We now give
\begin{proof}[Proof of Theorem \ref{main}]
By \ref{infinite} we may assume that the residue field of $A$ is infinite. If $G(A)$ is \CM \ then the result holds by Theorem 1.1 in \cite{ABIM}. So assume that $G(A)$ is not \CM. We prove $\ind(A) \geq \vt(A)$. By \ref{eqn} and \ref{ineq} this will imply the result.

Let $\bx = x_1,\ldots,x_d$ be an $A$-superficial sequence with $\vt(A) = \vt(A,\bx)$. Suppose if possible $\ind(A) < \vt(A,\bx)$. Say $\delta_A(A/\m^s) \geq 1$ for some $s < \vt(A,\bx)$. Set $A_0 = A$ and $A_i = A/(x_1,\ldots,x_i)$. Let $\m_i$ be the maximal ideal of $A_i$. We prove by descending induction that 
$$\delta_{A_i}(A_i/\m^s_i) \geq 1 \quad \text{for all} \ i.$$
For $i = 0$ this is our assumption. Now assume this is true for $i$. We prove it for $i+1$.  We first note that $s < \vt(A,\bx) \leq \vt(A_i,x_{i+1})$. Therefore
$(\m_i^{j+1} \colon x_{i+1}) = \m^j_i$ for all $j \leq s$. So 
by \ref{ding} we have that
$A_i/\m_i^s$ is an epimorphic image of $(\m_i^s,x)$. Thus $\delta_{A_i}((\m_i^s,x)) \geq 1$. We also have an $A_i$-module decomposition
\begin{equation*}
\frac{(\m_i^s,x_{i+1})}{x_{i+1}(\m^s_i,x_{i+1})} \cong \frac{A_i}{(\m^s_i,x_{i+1})} \oplus \frac{(\m^s_i, x_{i+1})}{(x_{i+1})}. \tag{$\dagger$}
\end{equation*}
By \ref{prop}(5) we have that 
\[
 \delta_{A_{i+1}}\left( \frac{(\m_i^s,x_{i+1})}{x_{i+1}(\m^s_i,x_{i+1})}\right) = \delta_{A_i}((\m_i^s,x)) \geq 1.
\]
Also note that
\[
\delta_{A_{i+1}}\left( \frac{(\m^s_i, x_{i+1})}{(x_{i+1})}\right) = \delta_{A_{i+1}}(\m_{i+1}^s) = 0, \ \text{by \ref{prop}(7)}.
\]
By $(\dagger)$ and 2.3(2) it follows that
\[
1 \leq \delta_{A_{i+1}}\left(\frac{A_i}{(\m^s_i,x_{i+1})} \right) = \delta_{A_{i+1}}\left( \frac{A_{i+1}}{\m^s_{i+1}} \right).
\]
This proves our inductive step. So we have $\delta_{A_d}(A_d/\m_d^s) \geq 1$. By \ref{prop}(6) we have that $\m_d^s = 0$. It follows that $\m^s \subseteq (\bx)$. This contradicts \ref{crucial}.
\end{proof}


\begin{thebibliography}{10}

\bibitem{AD}
M.~Auslander, A.~Ding, and O.~Solberg, 
\emph{Liftings and weak liftings of modules},
J. Algebra 156 (1993), no. 2, 273–-317. 

\bibitem{ABIM}
L.~Avramov,  R-O.~Buchweitz,  S.~B.~Iyengar, and C.~Miller, 
\emph{Homology of perfect complexes}, 
Adv. Math. 223 (2010), no. 5, 1731–-1781. 

\bibitem{D}
S.~Ding,
\emph{The associated graded ring and the index of a Gorenstein local ring},
Proc. Amer. Math. Soc. 120 (1994), no. 4, 1029–-1033. 

\bibitem{HM}
S.~Huckaba and T.~Marley,
\emph{Hilbert coefficients and the depths of associated graded rings},
J. London Math. Soc. (2) 56 (1997), no. 1, 64-–76. 


\bibitem{LW}
G.~J.~Leuschke and R.~Wiegand, 
\emph{Cohen-Macaulay representations},
Mathematical Surveys and Monographs, 181. American Mathematical Society, Providence, RI, 2012.

\bibitem{M}
A.~Martsinkovsky, 
\emph{New homological invariants for modules over local rings. I},
J. Pure Appl. Algebra 110 (1996), no. 1, 1–-8. 

\bibitem{M-r}
A.~Martsinkovsky, 
\emph{A remarkable property of the (co) syzygy modules of the residue field of a nonregular local ring},
J. Pure Appl. Algebra 110 (1996), no. 1, 9–-13. 

\bibitem{Pf}
T.~J.~Puthenpurakal, 
\emph{Hilbert-coefficients of a Cohen-Macaulay module},
J. Algebra 264 (2003), no. 1, 82–-97. 

\bibitem{P}
\bysame,
\emph{Ratliff-Rush filtration, regularity and depth of higher associated graded modules},
J. Pure Appl. Algebra 208 (2007), no. 1, 159–-176.

\bibitem{RV}
M.~E.~Rossi and G.~Valla, 
\emph{Cohen-Macaulay local rings of dimension two and an extended version of a conjecture of J. Sally}, 
J. Pure Appl. Algebra 122 (1997), no. 3, 293–-311

\bibitem{Sa}
J.~D.~Sally, 
\emph{Stretched Gorenstein rings},
J. London Math. Soc. (2) 20 (1979), no. 1, 19-–26. 


\bibitem{SS}
A-M.~Simon and J.~R.~Strooker,
\emph{Reduced Bass numbers, Auslander's δ-invariant and certain homological conjectures},
J. Reine Angew. Math. 551 (2002), 173–-218. 

\end{thebibliography}
\end{document}